\documentclass[12pt,a4paper]{article}
\usepackage{ifthen,latexsym,amssymb,amsmath}
\usepackage{amsfonts}
\usepackage{amsthm}
\usepackage{bbm}
\usepackage{url}
\usepackage{enumitem}
\usepackage{bm}
\usepackage[nottoc]{tocbibind}
\usepackage{graphicx}

\DeclareMathAlphabet\EuScript{U}{eus}{m}{n}
\SetMathAlphabet\EuScript{bold}{U}{eus}{b}{n}

\DeclareSymbolFont{rsfs}{U}{rsfs}{m}{n}
\DeclareSymbolFontAlphabet{\mathrsfs}{rsfs}

\theoremstyle{plain}
\newtheorem{theorem}{Theorem}
\newtheorem{lemma}[theorem]{Lemma}
\newtheorem{proposition}[theorem]{Proposition}
\newtheorem{corollary}[theorem]{Corollary}

\newtheorem*{theorem*}{Theorem} 

\theoremstyle{remark}

\newcommand{\R}{\mathbb{R}}

\newcommand{\E}{\mathbb{E}}
\newcommand{\e}{\varepsilon}

\newcommand{\one}{\mathbbm{1}}
\newcommand{\prob}{\mathbb{P}}

\newif\ifnotesw\noteswtrue

\noteswfalse	
\newcommand{\hide}[1]{}
\newcommand{\I}[1]{\mathbbm{#1}}

\title{Spherical sets avoiding a prescribed set of angles}
\author{Evan DeCorte\footnote{Supported by ERC advanced grant GA 320924-ProGeoCom
and at the beginning of the writing of this article by NWO Vidi grant 639.032.917.}\\
Einstein Institute of Mathematics\\
Hebrew University of Jerusalem\\
Givat Ram Jerusalem, 91904, Israel  \and Oleg Pikhurko\footnote{Supported by ERC
grant~306493 and EPSRC grant~EP/K012045/1.
 }\\
Mathematics Institute and DIMAP\\
University of Warwick\\
Coventry CV4 7AL, UK}
\date{\today}

\begin{document}

\maketitle

\abstract{Let $X$ be any subset of the interval $[-1,1]$.
A subset $I$ of the unit sphere in $\R^n$ will be called \emph{$X$-avoiding} if $\langle u,v \rangle \notin X$
for any $u,v \in I$. The problem of determining the maximum surface measure of a $\{ 0 \}$-avoiding
set was first stated in a 1974 note by Witsenhausen; there the upper bound of $1/n$ times the surface measure of the sphere is derived from a simple averaging argument. A consequence of the Frankl-Wilson theorem is that this fraction decreases exponentially, but until now the $1/3$ upper bound for the case $n=3$ has not moved. We improve this bound to $0.313$ using an approach inspired by Delsarte's linear programming bounds for codes, combined with some combinatorial reasoning. In the second part of the paper, we use harmonic
analysis to show that for $n\geq 3$
there always exists an $X$-avoiding set of maximum measure. We also show with
an example that a maximiser need not exist when $n=2$.}

\setlength{\parindent}{0mm}
\setlength{\parskip}{2mm}


\section{Introduction}

Witsenhausen \cite{witsenhausen74} in 1974 presented the following problem:
Let $S^{n-1}$ be the unit sphere in $\R^n$ and suppose $I \subset S^{n-1}$
is a Lebesgue measurable set such that no two vectors in $I$ are orthogonal. What is the largest possible Lebesgue surface measure of $I$?
Let $\alpha(n)$ denote the supremum of the measures of such sets $I$,
divided by the total measure of $S^{n-1}$.
The first upper bounds for $\alpha(n)$ appeared in
\cite{witsenhausen74}, where Witsenhausen deduced that $\alpha(n) \leq 1/n$.
In \cite{frankl-wilson81} Frankl and Wilson  proved their powerful combinatorial result on
intersecting set systems, and
as an application they gave the first exponentially decreasing upper bound
$\alpha(n) \leq (1+o(1))(1.13)^{-n}$.
Raigorodskii \cite{raigorodskii99} improved the bound to $(1+o(1))(1.225)^{-n}$
using a refinement of the Frankl-Wilson method.
Gil Kalai conjectured in his weblog \cite{kalai09} that an extremal example
is to take two opposite caps, each of geodesic radius $\pi/4$; if true, this implies that
$\alpha(n) = (\sqrt{2} + o(1))^{-n}$.

Besides being of independent interest, the above \emph{Double Cap Conjecture}
is also important
because, if true, it would imply new lower bounds for the
measurable chromatic number of Euclidean space, which we now discuss.

Let $c(n)$ be the smallest integer $k$ such that $\R^n$ can be partitioned into sets $X_1, \dots, X_k$,
with $\|x-y\|_2 \neq 1$ for each $x,y \in X_i$, $1 \leq i \leq k$. The number $c(n)$ is
called the \emph{chromatic number of $\R^n$}, since the sets $X_1,\dots,X_k$ can
be thought of as colour classes for a proper colouring of the graph on the vertex set $\R^n$, in which
we join two points with an edge when they have distance $1$.
Frankl and Wilson \cite[Theorem~3]{frankl-wilson81} showed that $c(n) \geq (1+o(1))(1.2)^n$,
proving a conjecture of Erd\H{o}s that $c(n)$ grows exponentially.
Raigorodskii in 2000 \cite{raigorodskii00}
improved the lower bound to $(1+o(1))(1.239)^n$.
Requiring the classes $X_1,\dots,X_k$ to be Lebesgue measurable yields the
\emph{measurable chromatic number} $c_m(n)$. Clearly $c_m(n) \geq c(n)$. Remarkably,
it is still open if the inequality is strict for at least one $n$, although one can prove better lower bounds
on $c_m(n)$. In particular, the exponent in Raigorodskii's bound
was recently beaten by Bachoc, Passuello and Thiery \cite{bachoc14} who showed that $c_m(n) \geq (1.268+o(1))^n$. If the Double Cap Conjecture is true,
then $c_m(n)\ge (\sqrt{2}+o(1))^n$ because, as it is not hard to show,
$c_m(n) \geq 1/\alpha(n)$ for every $n\ge 2$.
Note that the best known asymptotic upper bound 
on $c_m(n)$ (as well as on $c(n)$) is $(3+o(1))^n$, by Larman and Rogers~\cite{larman+rogers:72}. 

Despite progress on the asymptotics of $\alpha(n)$, the upper bound of $1/3$ for $\alpha(3)$
has not been improved since the original statement of the problem in \cite{witsenhausen74}.
Note that the two-cap construction gives $\alpha(3)\ge 1-1/\sqrt{2}=0.2928...$\ .
Our first main result is that $\alpha(3) < 0.313$. The proof involves tightening a Delsarte-type
linear programming upper bound (cf.\ \cite{delsarte73}, \cite{delsarte77}, \cite{bachoc09}, \cite{oliveira09}) by adding combinatorial constraints.

Let $\mathcal{L}$ be the $\sigma$-algebra of Lebesgue
surface measurable subsets of $S^{n-1}$, and let $\lambda$ be the surface
measure, for simplicity normalised so that $\lambda(S^{n-1}) = 1$.
For $X \subset [-1,1]$,
a subset $I \subset S^{n-1}$ will be called \emph{$X$-avoiding} if
$\langle \xi, \eta \rangle \notin X$ for all $\xi, \eta \in I$, where $\langle \xi, \eta \rangle$ denotes the
standard inner product of the vectors $\xi, \eta$. The corresponding extremal problem
is to determine
\begin{align}\label{eq:ir}
	\alpha_X(n):=\sup \{ \lambda(I)~:~\text{$I \in \mathcal{L}$, $I$ is $X$-avoiding} \}.
\end{align}

For example, if $t \in (-1,1)$ and $X=[-1,t)$, then $I \subset S^{n-1}$ is $X$-avoiding if and only if its geodesic diameter
is at most $\arccos(t)$. Thus Levy's isodiametric inequality~\cite{levy:pcaf} shows that $\alpha_X$
is given by a spherical cap of the appropriate size.

A priori, it is not clear that the value of $\alpha_X(n)$ 
is actually attained by some measurable $X$-avoiding set $I$ 
(so Witsenhausen \cite{witsenhausen74} had to use supremum to 
define $\alpha(n)$).
We prove in Theorem~\ref{thm:attainment} that the supremum is attained as a maximum
whenever $n \geq 3$. Remarkably,
this result holds under no additional assumptions whatsoever on the set $X$. However, in a sense
only closed sets $X$ matter: our Theorem~\ref{thm:closureIR} shows that $\alpha_X(n)$ does not change if we replace
$X$ by its closure. When $n=2$ the conclusion of Theorem~\ref{thm:attainment} fails; that is,
the supremum in \eqref{eq:ir} need not be a maximum: an example is given in Theorem
\ref{thm:irCircle}.

Besides also answering a natural question, the importance of
the attainment result can also be seen through the historical lens: In 1838 Jakob Steiner tried
to prove that a circle maximizes the area among all plane figures having some given perimeter.
He showed that any non-circle could be improved, but he was not able to
rule out the possibility that a sequence of ever improving plane shapes of equal perimeter could have
areas approaching some supremum which is not achieved as a maximum.
Only 40 years later in 1879 was the proof
completed, when Weierstrass showed that a maximizer must indeed exist.

The layout of the paper will be as follows. In Section \ref{sec:prelim} we make some general
definitions and fix notation. In Section \ref{sec:comb} we prove a simple and general proposition
giving combinatorial upper bounds for $\alpha_X(n)$; this is basically a formalisation of
the method used
by Witsenhausen in \cite{witsenhausen74} to obtain the $\alpha(n) \leq 1/n$ bound.
We then apply the proposition to calculate $\alpha_X(2)$ when $|X|=1$.
In Section \ref{sec:lp} we deduce linear programming
upper bounds for $\alpha(n)$, in the spirit of the Delsarte bounds for binary
\cite{delsarte73} and spherical \cite{delsarte77} codes.
We then strengthen the linear programming bound in the $n=3$ case
in Section \ref{sec:lp+comb} to obtain the first main result.
In Section \ref{sec:max} we prove
that the supremum $\alpha_X(n)$ is a maximum when $n \geq 3$, and in
Section \ref{sec:closure}
we show that $\alpha_X(n)$ remains unchanged when $X$ is replaced with its
topological closure. In Section \ref{sec:single} we formulate a conjecture
generalising the Double Cap Conjecture for the sphere in $\R^3$, in which other
forbidden inner products are considered.

\section{Preliminaries}\label{sec:prelim}

If $u,v \in \R^n$ are two vectors, their standard inner product will be denoted $\langle u, v \rangle$.
All vectors will be assumed to be column vectors.
The transpose of a matrix $A$ will be denoted $A^t$.
We denote by $SO(n)$ the group of $n \times n$ matrices $A$ over $\R$
having determinant $1$, for which $A^t A$ is equal to the identity matrix.
We will think of $SO(n)$ as a compact topological group, and we will always
assume its Haar measure is normalised so that $SO(n)$ has measure $1$.
We denote by $S^{n-1}$ the set of unit vectors in $\R^n$,
\[
	S^{n-1} = \{ x \in \R^n : \langle x, x \rangle = 1 \},
\]
equipped with its usual topology. The Lebesgue measure $\lambda$ on $S^{n-1}$
is always taken to be normalised so that $\lambda(S^{n-1}) = 1$.
Recall that the standard surface measure of $S^{n-1}$ is
 \begin{equation}\label{eq:OmegaN}
 \omega_n = \frac{2 \pi^{n/2}}{\Gamma(n/2)},
 \end{equation} where $\Gamma$ denotes Euler's gamma-function.
The Lebesgue $\sigma$-algebra on $S^{n-1}$ will be denoted by $\mathcal{L}$.
When $(X, \mathcal{M}, \mu)$ is a measure space and $1 \leq p < \infty$,
we use
\[
	L^p(X) = \left\{ f :~\text{$f$ is an $\R$-valued $\mathcal{M}$-measurable function and
	$\int |f|^p\,\mathrm{d}\mu < \infty$} \right\}.
\]

For $f \in L^p(X)$, we define $\|f\|_p := \left( \int |f|^p\,\mathrm{d}\mu \right)^{1/p}$. Identifying
two functions when they agree $\mu$-almost everywhere, we make $L^p(X)$ 
a Banach space with the norm $\| \cdot \|_p$.

We will use bold letters (for example $\bm{X}$) for random variables.
The expectation of a function $f$ of a random variable $\bm{X}$ will be denoted
$\E_{\bm{X}}[f(\bm{X})]$, or just $\E[f(\bm{X})]$. The probability of an event $E$
will be denoted $\prob[E]$.

When $X$ is a set, we use $\one_X$ to denote its characteristic function; that is
$\one_X(x) = 1$ if $x \in X$ and $\one_X(x) = 0$ otherwise.




If $G = (V,E)$ is a graph, a set $I$ is called \emph{independent}
if $\{u,v\} \notin E$ for any $u,v \in I$. The \emph{independence number} $\alpha(G)$ of $G$
is the cardinality of a largest independent set in $G$. We define
$\alpha_X(n)$ as in \eqref{eq:ir}, and for brevity we let $\alpha(n) = \alpha_{\{0\}}(n)$.


\section{Combinatorial upper bound}\label{sec:comb}

Let us begin by deriving a simple ``combinatorial'' upper bound for the quantity $\alpha_X(n)$.



\begin{proposition}\label{pr:comb}
	Let $n \geq 2$ and $X \subset [-1,1]$. For a finite subset $V \subset S^{n-1}$,
	we let $H=(V,E)$ be the graph on the vertex set $V$ with edge set defined by
	putting $\{ \xi, \eta \} \in E$ if and only if $\langle \xi, \eta \rangle \in X$.
	Then $\alpha_X(n) \leq \alpha(H) / |V|$.
\end{proposition}

\begin{proof} Let $I \subset S^{n-1}$ be an $X$-avoiding set, and
take a uniform $\bm{O}\in SO(n)$. Let the random variable $\bm{Y}$ be the number of
$\xi\in V$ with $\bm{O}\xi\in I$. Since $\bm{O}\xi\in S^{n-1}$ is uniformly distributed
for every $\xi\in V$, we have by the linearity of expectation that $\E(\bm{Y})=|V|\, \lambda(I)$.
On the other hand, $\bm{Y} \le \alpha(H)$ for every outcome $\bm{O}$.
Thus $\lambda(I) \leq \alpha(H)/|V|$.
\end{proof}


We next use Proposition~\ref{pr:comb} to find the largest possible Lebesgue measure
of a subset of the unit circle in $\R^2$ in which no two points lie at some fixed
forbidden angle. 

\begin{theorem}\label{thm:irCircle}
Let $X = \{x\}$ and put $t = \frac{\arccos{x}}{2\pi}$.
If $t$ is rational and $t = p/q$ with $p$ and $q$ coprime integers, then
\begin{align*}
	\alpha_X(2) = \begin{cases}
		1/2, &~\text{if $q$ is even,}\\
		(q-1)/(2q), &~\text{if $q$ is odd.}
	\end{cases}
\end{align*}
In this case $\alpha_X(2)$ is attained as a maximum. If $t$ is irrational then
$\alpha_X(2) = 1/2$, but there exists no measurable $X$-avoiding set $I \subset S^1$
with $\lambda(I) = 1/2$.
\end{theorem}


\begin{proof}
Write $\alpha = \alpha_X(2)$, and identify $S^1$ with the interval $[0,1)$
via the map $(\cos x, \sin x) \mapsto x/2\pi$.
We regard $[0,1)$ as a group with the operation of addition modulo $1$.
Notice that $I \subset [0,1)$ is $X$-avoiding if and only if $I \cap (t+I) = \emptyset$.
This implies immediately that $\alpha \leq 1/2$ for all values of~$x$.

Now suppose $t = p/q$ with $p$ and $q$ coprime integers, and
suppose that $q$ is even.
Let $S$ be any open subinterval of $[0,1)$ of length $1/q$, and
define $T : [0,1) \to [0,1)$ by $T x = x+t \mod 1$.
Using the fact that $p$ and $q$ are coprime,
one easily verifies that $I = S \cup T^2 S \cup \cdots \cup T^{q-4} S \cup T^{q-2} S$
has measure $1/2$. Also $S$ is $X$-avoiding since
$T S = T S \cup T^3 S \cup \cdots \cup T^{q-3} S \cup T^{q-1} S$
is disjoint from $S$. Therefore $\alpha = 1/2$.

Next suppose $q$ is odd.
With notation as before, a similar argument shows that
$I \cup T^2 I \cup \cdots \cup T^{q-3} I$
is an $X$-avoiding set of measure $(q-1)/(2q)$,
and Proposition~\ref{pr:comb} shows that this is largest possible,
since the points $x, T x, T^2 x, \dots, T^{q-1} x$ induce a $q$-cycle.

Finally suppose that $t$ is irrational. By Dirichlet's approximation theorem
there exist infinitely many pairs of coprime integers $p$ and $q$ such that
$| t - p/q | < 1/q^2$. For each such pair, let $\e = \e(q) = | t - p/q |$.
Using an open interval $I$ of length $\frac{1}{q} - \e$ and applying
the same construction as above with $T$ defined by $Tx = x+p/q$, one obtains an
$X$-avoiding set of measure at least
$((q-1)/2)(1/q - \e) = 1/2 - o(1)$. Alternatively,
the lower bound $\alpha \ge 1/2$ follows from 
Rohlin's tower theorem (see e.g.\ \cite[Theorem~169]{kalikow+mccutcheon:oet}) applied to the ergodic transformation $Tx = x+t$. Therefore $\alpha = 1/2$.

However this supremum can never be attained. Indeed, if $I \subset [0,1)$ is
an $X$-avoiding set with $\lambda(I) = 1/2$ and $T$ is defined by $Tx = x+t$, then
$I \cap T I = \emptyset$ and $T I \cap T^2 I = \emptyset$. Since $\lambda(I)=1/2$,
this implies that $I$ and $T^2 I$ differ only on a nullset,
contradicting the ergodicity of the irrational rotation $T^2$.
\end{proof}

\section{Gegenbauer polynomials and Schoenberg's theorem}\label{sec:gegenbauer}

Before proving the first main result, we recall the Gegenbauer polynomials
and Schoenberg's theorem from the theory of spherical harmonics.
For $\nu > -1/2$, define the \emph{Gegenbauer weight function}\index{Gegenbauer weight function}
\[
	r_\nu(t) := (1-t^2)^{\nu-1/2},~~ -1 < t < 1.
\]
 To motivate this definition, observe that if we take a uniformly distributed vector $\bm{\xi}\in S^{n-1}$, $n \geq 2$,
and project it to any given axis, then the density of the obtained random variable $\bm{X}\in[-1,1]$ is proportional to
$r_{(n-2)/2}$, with the coefficient
$\left( \int_{-1}^1 r_{(n-2)/2}(x)\,\mathrm{d}x \right)^{-1} =\frac{\omega_{n-1}}{\omega_n}$
where $\omega_n$ is as in (\ref{eq:OmegaN}). (In particular, $\bm{X}$ is 
uniformly distributed in $[-1,1]$ if $n=3$.)

Applying the Gram-Schmidt process to the polynomials $1,t,t^2,\dots$ with respect to the inner
product $\langle f, g \rangle_\nu = \int_{-1}^1 f(t) g(t) r_\nu(t)\,\mathrm{d}t$, one obtains
the \emph{Gegenbauer polynomials}\index{Gegenbauer polynomials}
$C_i^\nu(t)$, $i=0,1,2,\dots$, where $C_i^\nu$ is of degree $i$.  For a concise
overview of these polynomials, see e.g.\ \cite[Section B.2]{dai+xu:athasb}. Here, we always use the normalisation
$C_i^\nu(1)=1$.

For a fixed $n \geq 2$, a continuous function $f : [-1,1] \to \R$ is called
\emph{positive definite}\index{positive definite!function on $S^{n-1}$}
if for every set of distinct points $\xi_1, \dots, \xi_s \in S^{n-1}$, the
matrix $(f(\langle \xi_i, \xi_j \rangle))_{i,j=1}^s$ is positive semidefinite.
We will need the following result of Schoenberg~\cite{shoenberg:38}
(for a modern presentation, see e.g.\ \cite[Theorem~14.3.3]{dai+xu:athasb}).

\begin{theorem}[Schoenberg's theorem]\label{thm:schoenberg}\index{Schoenberg's theorem}
	For $n \geq 2$, a continuous function $f : [-1,1] \to \R$ is positive definite if and only if
	there exist coefficients $a_i \geq 0$, for $i \geq 0$, such that
	\[
		f(t) = \sum_{i=0}^\infty a_i C_i^{(n-2)/2}(t),~~~~\text{for all}~t \in [-1,1].
	\]
Moreover,  the convergence on the right-hand side is absolute and uniform for every
positive definite function $f$.
\end{theorem}

For a given positive definite function $f$, the coefficients $a_i$ in
Theorem~\ref{thm:schoenberg} are unique and can be computed
explicitly; a formula is given in \cite[Equation (14.3.3)]{dai+xu:athasb}.

We are especially interested in the case $n=3$. Then $\nu = 1/2$, and the
first few Gegenbauer polynomials $C_i^{1/2}(x)$ are
\begin{gather*}
	C_0^{1/2}(x) = 1,~~~C_1^{1/2}(x) =x,~~~C_2^{1/2}(x) = \frac{1}{2} \left( 3x^2 - 1\right),\\
	C_3^{1/2}(x) = \frac{1}{2}\left( 5x^3 - 3x \right),
	~~~C_4^{1/2}(x) = \frac{1}{8}\left( 35x^4 -30x^2+3 \right).
\end{gather*}

\section{Linear programming relaxation}\label{sec:lp}

Schoenberg's theorem allows us to
set up a linear program whose value upper bounds
$\alpha(n)$ for $n \geq 3$. The same result appears in
\cite{bachoc09} and \cite{oliveira09}; we present a self-contained
(and slightly simpler) proof for the reader's convenience.
In the next section we strengthen the linear program,
obtaining a better bound for~$\alpha(3)$.

\begin{lemma}\label{lm:pt}
	Suppose $f,g \in L^2(S^{n-1})$ and define $k : [-1,1] \to \R$ by
	\begin{align}\label{eq:correlation}
		k(t) := \E[f(\bm{O} \xi) g(\bm{O} \eta)],
	\end{align}
	where the expectation is taken over randomly chosen $\bm{O} \in SO(n)$, 
	and $\xi,\eta \in S^{n-1}$ are any two points satisfying $\langle \xi, \eta \rangle = t$.
	Then $k(t)$ exists for every $-1 \leq t \leq 1$, and
	$k$ is continuous. If $f=g$, then $k$ is positive definite.
\end{lemma}
\begin{proof}
	The expectation in \eqref{eq:correlation} clearly does not depend on the particular
	choice of $\xi,\eta \in S^{n-1}$. Fix any point $\xi_0 \in S^{n-1}$ and let
	$P : [-1,1] \to SO(n)$ be any continuous function satisfying
	$\langle \xi_0, P(t) \xi_0 \rangle = t$ for each $-1 \leq t \leq 1$. We have
	\begin{align}\label{eq:correlation1}
		k(t) = \E[f(\bm{O}\xi_0) g(\bm{O}P(t)\xi_0)].
	\end{align}
	The functions $O \mapsto f(O\xi_0)$ and $O \mapsto g(OP(t)\xi_0)$ on $SO(n)$
	belong to $L^2(SO(n))$; being an inner product in $L^2(SO(n))$, the expectation
	\eqref{eq:correlation1} therefore exists for every $t \in [-1,1]$.
	
	We next show that $k$ is continuous.
	For each  $O \in SO(n)$, let $R_O : L^2(SO(n)) \to L^2(SO(n))$ be the
	\emph{right translation} operator defined by $(R_O F)(O') = F(O' O)$,
	for $F \in L^2(SO(n))$. For fixed $F$, the map $O \mapsto R_O F$ is continuous
	from $L^2(SO(n))$ to $L^2(SO(n))$ (see e.g. \cite[Lemma~1.4.2]{deitmar09}).
	Therefore the function $t \mapsto R_{P(t)} F$ is continuous from $[-1,1]$
	to $L^2(SO(n))$. Using $F(O) = g(O\xi_0)$, the continuity of $k$ follows.
	
	Now suppose $f=g$; we show that $k$ is positive definite. Let $\xi_1, \dots, \xi_s \in S^{n-1}$.
	We need to show the $s \times s$ matrix $K = (k(\xi_i, \xi_j))_{i,j=1}^s$ is positive
	semidefinite. But if $v = (v_1, \dots, v_s)^T \in \R^s$ is any column vector, then
	\begin{align*}
		v^T K v &= \sum_{i=1}^s \sum_{j=1}^s \E[f(\bm{O} \xi_i) f(\bm{O} \xi_j)] v_i v_j
			= \E \left[ \left( \sum_{i=1}^s f(\bm{O} \xi_i)  v_i \right)^2 \right] \geq 0.
	\end{align*}
\end{proof}

\begin{theorem}\label{thm:lp-upper-bound}
	$\alpha(n)$ is no more than the
	value of the following infinite-dimensional linear program.
	\begin{align}\label{eq:lpweak}
	\begin{gathered}
	\max x_0\\
	\sum_{i=0}^\infty x_i = 1\\
	\sum_{i=0}^\infty x_i C_i^{(n-2)/2}(0) = 0\\
	x_i \geq 0,~\text{for all $i = 0,1,2,\dots$}~.
	\end{gathered}
	\end{align}
\end{theorem}
\begin{proof}
Let $I \in \mathcal{L}$ be a $\{0\}$-avoiding subset of $S^{n-1}$ with $\lambda(I)>0$.
We construct a feasible solution to the linear program \eqref{eq:lpweak}
having value $\lambda(I)$.
Let $k : [-1,1] \to \R$ be defined as in \eqref{eq:correlation}, with $f=g=\one_I$.
Then $k$ is a positive definite function satisfying $k(1) = \lambda(I)$ and $k(0) = 0$.
By Theorem~\ref{thm:schoenberg},
$k$ has an expansion in terms of the Gegenbauer polynomials:
\begin{align}\label{eq:gegenbauer}
	k(t) = \sum_{i=0}^\infty a_i C_i^{(n-2)/2}(t),
\end{align}
where each $a_i\ge 0$ and the convergence of the right-hand side is uniform on $[-1,1]$.
Moreover, for each fixed $\xi_0 \in S^{n-1}$, we have by Fubini's theorem and
\eqref{eq:correlation} that
\begin{align}\label{eq:gegenbauerInt}
	\int_{S^{n-1}} k(\langle \xi_0, \eta \rangle) \,\mathrm{d}\eta
	&= \int_{S^{n-1}} \int_{S^{n-1}} k(\langle \xi, \eta \rangle) \, \mathrm{d} \xi \,\mathrm{d}\eta\\
	&= \E \left[ \left( \int_{S^{n-1}} \one_I(\bm{O} \xi) \, \mathrm{d}\xi \right)^2 \right]
	= \lambda(I)^2.
\end{align}
Note that
\begin{align*}
	\int_{S^{n-1}} C_i^{(n-2)/2}(\langle \xi_0, \eta \rangle) \, \mathrm{d}\eta
	= \frac{\omega_{n-1}}{\omega_n} \int_{-1}^1 C_i^{(n-2)/2}(t) (1-t^2)^{(n-3)/2}\,\mathrm{d}t = 0
\end{align*}
whenever $i \geq 1$ by the definition of the Gegenbauer polynomials.
Putting \eqref{eq:gegenbauer} and \eqref{eq:gegenbauerInt} together
and using that $C_0^{(n-2)/2} \equiv 1$, we conclude
that $a_0 = \lambda(I)^2$.

Recalling that $C_i^{(n-2)/2}(1) = 1$ for $i \geq 0$, we find that
setting $x_i = a_i / \lambda(I)$ for $i=0,1,2,\dots$ gives a feasible solution of
value $\lambda(I)$ to the linear program \eqref{eq:lpweak}.\qedhere

\end{proof}

Unfortunately in the case $n=3$, the value of \eqref{eq:lpweak} is at least $1/3$,
which is the same bound obtained when Witsenhausen first stated
the problem in \cite{witsenhausen74}. This can be seen from the feasible solution
$x_0 = 1/3, x_2=2/3$ and $x_i = 0$ for all $i \neq 0,2$.

\section{Adding combinatorial constraints}\label{sec:lp+comb}

For each $\xi \in S^{n-1}$ and $-1 < t < 1$, let $\sigma_{\xi,t}$ be the unique
probability measure on the Borel subsets of $S^{n-1}$ whose support is equal to the set
$$
 \xi^t := \{ \eta \in S^{n-1} : \langle \eta, \xi \rangle = t \},
  $$ and which is invariant under all
rotations fixing $\xi$. 

Now let $n=3$.
As before, let $I \in \mathcal{L}$ be a $\{0\}$-avoiding subset of $S^2$ and define $k : [-1,1] \to \R$
as in \eqref{eq:correlation} with $f=g=\one_I$; i.e.
\[
	k(t) = \E[\one_I(\bm{O}\xi) \one_I(\bm{O}\eta)],
\]
where $\xi, \eta \in S^2$ satisfy $\langle \xi, \eta \rangle = t$.

Our aim now is to strengthen \eqref{eq:lpweak} for the case $n=3$
by adding combinatorial inequalities coming from Proposition~\ref{pr:comb}
applied to the sections of $S^2$ by affine planes.
We proceed as follows.
Let $p$ and $q$ be coprime integers with $1/4 \leq p/q \leq 1/2$, and let
 $$
 t_{p,q} = \sqrt{\frac{-\cos(2\pi p/q)}{1-\cos(2\pi p/q)}}.$$
Let $\xi \in S^{n-1}$ be arbitrary.
If we take two orthogonal unit vectors with endpoints in $\xi^{t_{p,q}}$ and
the centre $\xi_0=t_{p,q}\xi$ of this circle, then we get an isosceles triangle with side lengths $(1-t_{p,q}^2)^{1/2}$
and base $\sqrt{2}$; by the Cosine Theorem, the angle at $\xi_0$ is
$2\pi p/q$.

Let $\xi_0, \eta_0 \in S^{n-1}$ be arbitrary points satisfying $\langle \xi_0, \eta_0 \rangle = t_{p,q}$.
By Fubini's theorem we have
\begin{align*}
	k(t_{p,q}) &= \E[\one_I(\bm{O}\xi_0) \one_I(\bm{O}\eta_0)]
	= \int_{\xi_0^{t_{p,q}}} \E[\one_I(\bm{O}\xi_0) \one_I(\bm{O}\eta)]
	\, \mathrm{d}\sigma_{\xi_0,t_{p,q}}(\eta)\\
	&= \E \left[ \one_I(\bm{O}\xi_0) \int_{\xi_0^{t_{p,q}}} \one_I(\bm{O}\eta)
	\, \mathrm{d}\sigma_{\xi_0,t_{p,q}}(\eta) \right].
\end{align*}
But if $q$ is odd, then $\int_{\xi_0^{t_{p,q}}} \one_I(O \eta) \, \mathrm{d}
\sigma_{\xi_0,t_{p,q}}(\eta) \leq \frac{q-1}{2q}$ for all $O \in SO(n)$
by Proposition~\ref{pr:comb} applied to the circle $(O\xi_0)^{t_{p,q}}\cong S^1$,
since the subgraph it induces contains a cycle of length $q$.
Therefore $k(t_{p,q}) \leq \lambda(I) \frac{q-1}{2q}$.

It follows that the inequalities
\begin{align}\label{eq:combIneq}
	\sum_{i=0}^\infty x_i C_i^{1/2}(t_{p,q}) \leq (q-1)/2q,
\end{align}
are valid for the relaxation and can be added to \eqref{eq:lpweak}.
The same holds for the inequalities $\sum_{i=0}^\infty x_i C_i^{1/2}(-t_{p,q}) \leq (q-1)/2q$.

So we have just proved the following result.

\begin{theorem}
	$\alpha(3)$ is no more than the value of the following infinite-dimensional linear program.
	\begin{align}\label{eq:lp-super-strong}
	\begin{gathered}
	\max x_0\\
	\sum_{i=0}^\infty x_i = 1\\
	\sum_{i=0}^\infty x_i C_i^{1/2}(0) = 0\\
	\sum_{i=0}^\infty x_i C_i^{1/2}(\pm t_{p,q}) \leq (q-1)/2q,~~\text{for $q$ odd,~~$p,q$
	coprime}\\
	x_i \geq 0,~\text{for all $i = 0,1,2,\dots$}\ .
	\end{gathered}
	\end{align}
\end{theorem}

Rather than attempting to find the exact value of the linear program \eqref{eq:lp-super-strong},
the idea will be to discard all but finitely many of the combinatorial constraints, and
then to apply the weak duality theorem of linear programming. The dual linear program
has only finitely many variables, and any feasible solution gives an upper bound
for the value of program \eqref{eq:lp-super-strong}, and therefore also for $\alpha(3)$.
At the heart of the proof is the verification of the feasibility of a particular dual
solution which we give explicitly.
While part of the verification has been carried out by computer in order to deal with
the large numbers that appear, it can be done using only rational arithmetic and can therefore
be considered rigorous.

\begin{theorem}
	$\alpha(3) < 0.313$.
\end{theorem}
\begin{proof}

Consider the following linear program
\begin{align}\label{eq:lpstrong}
	\max \Big\{ x_0 &: \sum_{i=0}^\infty x_i = 1, \sum_{i=0}^\infty x_i C_i^{1/2}(0) = 0,
	\sum_{i=0}^\infty x_i C_i^{1/2}(t_{1,3}) \leq 1/3,\\
	&\sum_{i=0}^\infty x_i C_i^{1/2}(t_{2,5}) \leq 2/5,
	\sum_{i=0}^\infty x_i C_i^{1/2}(-t_{2,5}) \leq 2/5,\nonumber\\
	&~~~~~~~~~~~x_i \geq 0,~\text{for all $i = 0,1,2,\dots$}\nonumber  \Big\}.
\end{align}

The linear programming dual of \eqref{eq:lpstrong} is the following.
\begin{align}\label{eq:lpdual}
\begin{gathered}
	\min ~b_1 + \frac{1}{3}b_{1,3} + \frac{2}{5}b_{2,5} + \frac{2}{5}b_{2,5-}\\
	b_1 + b_0 + b_{1,3} + b_{2,5} + b_{2,5-} \geq 1\\
	b_1 + C_i^{1/2}(0) b_0 + C_i^{1/2}(t_{1,3}) b_{1,3} + C_i^{1/2}(t_{2,5}) b_{2,5}
		+ C_i^{1/2}(-t_{2,5}) b_{2,5-} \geq 0
	~\text{for $i = 1,2,\dots$}\\
	b_1, b_0 \in \R, ~b_{1,3}, b_{2,5}, b_{2,5-} \geq 0.
\end{gathered}
\end{align}

By linear programming duality,
any feasible solution for program \eqref{eq:lpdual} gives an upper bound
for \eqref{eq:lpstrong}, and therefore also for $\alpha(3)$.
So in order to prove the claim $\alpha(3) < 0.313$, it suffices
to give a feasible solution to \eqref{eq:lpdual} having objective value no more than $0.313$.
Let
	\begin{align*}
		b = (b_1, b_0, b_{1,3}, b_{2,5}, b_{2,5-})
		= \frac{1}{10^6}(128614, 404413, 36149, 103647, 327177).
	\end{align*}
	It is easily verified that $b$ satisfies the first constraint of \eqref{eq:lpdual}
	and that its objective value less than $0.313$.
	To verify the infinite family of constraints
	\begin{align}\label{eq:constraints}
		b_1 + C_i^{1/2}(0) b_0 + C_i^{1/2}(t_{1,3}) b_{1,3} + C_i^{1/2}(t_{2,5}) b_{2,5}
		+ C_i^{1/2}(-t_{2,5}) b_{2,5-} \geq 0
	\end{align}
	for $i=1,2,\dots$, we apply Theorem~8.21.11 from \cite{szego92} (where
$C_i^{\lambda}$ is denoted as $P_i^{(\lambda)}$), which implies
	\begin{align}\label{eq:gegenbauer-upper-bound}
		| C_i^{1/2}(\cos{\theta}) | \leq \frac{\sqrt{2}}{\sqrt{\pi} \sqrt{\sin{\theta}}}\,
			\frac{\Gamma(i+1)}{\Gamma(i+3/2)}
			+ \frac{1}{\sqrt{\pi} 2^{3/2} (\sin{\theta})^{3/2}}\, \frac{\Gamma(i+1)}{\Gamma(i+5/2)}
	\end{align}
	for each $0 < \theta < \pi$.
	Note that $t_{1,3} = 1/\sqrt{3}$ and $t_{2,5}=5^{-1/4}$.
	When $\theta \in A:= \{ \pi/2, \arccos{t_{1,3}}, \arccos{t_{2,5}}, \arccos{(-t_{2,5})} \}$,
	we have $\sin{\theta} \in \{ 1, \sqrt{\frac{2}{3}}, \gamma \}$, where
	$\gamma =\frac{2}{\sqrt{5+\sqrt{5}}}$. The right-hand side of equation
	\eqref{eq:gegenbauer-upper-bound} is maximized over $\theta\in A$ at $\sin{\theta} = \gamma$
	for each fixed $i$, and since the right-hand side is decreasing in $i$,
	one can verify using rational arithmetic only that it is no greater than
	$128614 / 871386 = b_1 / (b_0+b_{1,3}+b_{2,5}+b_{2,5-})$ when $i \geq 40$, by
	evaluating at $i=40$. Therefore,
	\begin{align*}
		&b_1 + C_i^{1/2}(0) b_0 + C_i^{1/2}(t_{1,3}) b_{1,3} + C_i^{1/2}(t_{2,5}) b_{2,5}
		+ C_i^{1/2}(-t_{2,5}) b_{2,5-}\\
		\geq&~ b_1 - (b_0+b_{1,3}+b_{2,5}+b_{2,5-})\max_{\theta \in A}\{ | C_i^{1/2}(\cos{\theta}) | \}\\
		\geq&~0
	\end{align*}
	when $i \geq 40$. It now suffices to check that $b$ satisfies the constraints
	\eqref{eq:constraints} for $i=0,1,\dots,39$. This can also be accomplished using
	rational arithmetic only.
\end{proof}

The rational arithmetic calculations required in the above proof were carried out with \emph{Mathematica}.
When verifying the upper bound for the right-hand side of \eqref{eq:gegenbauer-upper-bound},
it is helpful to recall the identity $\Gamma(i+1/2) = (i-1/2)(i-3/2) \cdots (1/2) \sqrt{\pi}$.
When verifying the constraints \eqref{eq:constraints} for $i=0,1,\dots,39$, it can be
helpful to observe that $t_{1,3}$ and $\pm t_{2,5}$ are roots of the polynomials
$x^2 - 1/3$ and $x^4-1/5$ respectively; this can be used to cut down the degree
of the polynomials $C_i^{1/2}(x)$ to at most $3$ before evaluating them.  The ancillary folder
of the \texttt{arxiv.org} version of this paper contains a \emph{Mathematica} notebook that
verifies all calculations.

The combinatorial inequalities
of the form \eqref{eq:combIneq} we chose to include in the strengthened linear program
\eqref{eq:lpstrong} were found as follows: Let $L_0$ denote the linear program \eqref{eq:lpweak}.
We first find an optimal solution $\sigma_0$ to $L_0$.
We then proceed recursively; having defined the linear program $L_{i-1}$ and found
an optimal solution $\sigma_{i-1}$,
we search through the inequalities \eqref{eq:combIneq} until
one is found which is violated by $\sigma_{i-1}$, and we strengthen $L_{i-1}$ with
that inequality to produce $L_i$. At each stage, an optimal solution to $L_i$
is found by first solving the dual minimisation problem, and then applying
the complementary slackness theorem from linear programming
to reduce $L_i$ to a linear programming
maximisation problem with just a finite number of variables.

Adding more inequalities of the form $\eqref{eq:combIneq}$ appears
to give no improvement on the upper bound.
Also adding the constraints $\sum_{i=0}^\infty x_i C_i^{1/2}(t) \geq 0$
for $-1\leq t \leq 1$ appears to give no improvement.
A small (basically insignificant)
improvement can be achieved by allowing the odd cycles to embed into $S^2$ in
more general ways, for instance with the points lying on two different latitudes rather
than just one.

\section{Adjacency operator}\label{sec:adj}

Let $n \geq 3$. For $\xi \in S^{n-1}$ and $-1 < t < 1$, we use the notations $\xi^t$
and $\sigma_{\xi,t}$ from Section \ref{sec:lp+comb}.
For $f \in L^2(S^{n-1})$ define $A_tf: S^{n-1}\to\I R$ by
\begin{align}\label{eq:transOperator}
	(A_t f)(\xi) := \int_{\xi^t} f(\eta)\,\mathrm{d}\sigma_{\xi,t}(\eta),\quad \xi\in S^{n-1}.
\end{align}

Here we establish some basic properties of $A_t$  which will be helpful later.
The operator $A_t$ can be thought of as an adjacency operator for the graph
with vertex set $S^{n-1}$, in which we join two points with an edge
when their inner product is $t$.
Adjacency operators for infinite graphs are explored in greater detail and generality in
\cite{bachoc13}.

\begin{lemma}\label{lm:adjacency}
	For every $t \in (-1,1)$, $A_t$ is a bounded linear operator from $L^2(S^{n-1})$ to
	$L^2(S^{n-1})$ having operator norm equal to $1$.
\end{lemma}
\begin{proof}
	The right-hand side of \eqref{eq:transOperator} involves integration
	over nullsets of a function $f \in L^2(S^{n-1})$ which is only defined almost everywhere,
	and so strictly speaking one should argue that \eqref{eq:transOperator} really makes sense.
	In other words, given a particular representative~$f$ from its
	$L^2$-equivalence class, we need to check that the integral on the right-hand side of
	\eqref{eq:transOperator} is defined for almost all $\xi \in S^{n-1}$,
	and that the $L^2$-equivalence class of $A_t f$ does not
	depend on the particular choice of representative~$f$.

	Our main tool will be Minkowski's integral inequality (see e.g.\  \cite[Theorem~6.19]{folland99}).
	
	Let $e_n = (0,\dots,0,1)$ be the $n$-th basis
	vector in $\R^n$ and let
	\[
		S = \{ (x_1, x_2, \dots, x_n) : x_n=0, x_1^2 + \dots + x_{n-1}^2 =1 \}
	\]
	be a copy of $S^{n-2}$ inside $\R^n$. Considering $f$ as a particular measurable function
	(not an $L^2$-equivalence class), we define $F : SO(n) \times S \to \R$ by
	\[
		F(\rho,\eta)
		=f\left (\rho\left(t e_n + \sqrt{1-t^2}\, \eta\right)\right),\qquad \rho \in SO(n),\ \eta\in S.
	\]
	Let us formally check all the hypotheses of Minkowski's integral inequality applied
	to $F$, where $SO(n)$ is equipped with the Haar measure, and where $S$ is
	equipped with the normalised Lebesgue measure;
	this will show that the function $\tilde{F} : SO(n) \to \R$ defined by
	$\tilde{F}(\rho) = \int_S F(\rho, \eta)\,\mathrm{d}\eta$
	belongs to $L^2(SO(n))$.
	
	Clearly the function $F$ is measurable.
	To see that the function $\rho \mapsto F(\rho, \eta)$ belongs to $L^2(SO(n))$
	for each fixed $\eta \in S$, simply note that
	\[
		\int_{SO(n)} \left| F(\rho,\eta) \right|^2 \,\mathrm{d}\rho
		= \int_{SO(n)} \left| f(\rho(te_n + \sqrt{1-t^2}\,\eta)) \right|^2 \,\mathrm{d}\rho
		= \|f\|_2^2.
	\]
	That the function $\eta \mapsto \| F(\cdot, \eta) \|_2$ belongs to $L^1(S)$
	then also follows easily (in fact, this function is constant):
	\[
		\int_S \left( \int_{SO(n)} \left| F(\rho,\eta) \right|^2 \,\mathrm{d}\rho \right)^{1/2} \,\mathrm{d}\eta
		= \int_S \|f\|_2 \,\mathrm{d}\eta
		= \|f\|_2.
	\]
	Minkowski's integral inequality now gives that the function
	$\eta \mapsto F(\rho, \eta)$ is in $L^1(S)$ for a.e.\ $\rho$,
	the function $\tilde{F}$ is in $L^2(SO(n))$, and its norm can be
	bounded as follows:
	\begin{align}
		\|\tilde{F}\|_2 &= 
		\left( \int_{SO(n)} \left| \int_S F(\rho,\eta) \,\mathrm{d}\eta \right|^2
		\,\mathrm{d}\rho\right)^{1/2}\nonumber\\
		&\leq \int_{S} \left( \int_{SO(n)} |F(\rho,\eta)|^2 \,\mathrm{d}\rho \right)^{1/2}  \,\mathrm{d}\eta
		= \|f\|_2.\label{eq:FNormBound}
	\end{align}
	Applying \eqref{eq:FNormBound} to $f-g$ where $g$ is a.e.\ equal to $f$,
	we conclude that the $L^2$-equivalence class of $\tilde{F}$ does
	not depend on the particular choice of representative $f$ from its equivalence class. 
	
	Now $(A_t f)(\xi)$ is simply $\tilde{F}(\rho)$, where $\rho \in SO(n)$
	can be any rotation such that $\rho e_n = \xi$. This shows that the
	integral in \eqref{eq:transOperator} makes sense for almost all $\xi \in S^{n-1}$.
	
	We have $\|A_t\| \leq 1$
	since for any $f \in L^2(S^{n-1})$,
	
	\begin{align*}
		\| A_t f \|_2 = \left( \int_{S^{n-1}} \left| (A_t f)(\xi) \right|^2 \,\mathrm{d}\xi \right)^{1/2}
		&=  \left( \int_{SO(n)} \left| (A_t f)(\rho e_n) \right|^2 \,\mathrm{d}\rho \right)^{1/2}\\
		&=  \left( \int_{SO(n)} \left| \tilde{F}(\rho) \right|^2 \,\mathrm{d}\rho \right)^{1/2}
		\leq \| f \|_2,
	\end{align*}
	by \eqref{eq:FNormBound}.
	
	Finallly, applying $A_t$
	to the constant function $1$ shows that $\|A_t\| = 1$.\qedhere
\end{proof}

\begin{lemma}\label{lm:twoPoint}
	Let $f$ and $g$ be functions in $L^2(S^{n-1})$, let $\xi, \eta \in S^{n-1}$
	be arbitrary points, and write $t = \langle \xi, \eta \rangle$.
	If $\bm{O} \in SO(n)$ is chosen uniformly at random with respect to the Haar measure
	on $SO(n)$, then
	\begin{align}\label{eq:k-t}
		\int_{S^{n-1}} f(\zeta) (A_t g)(\zeta) \,\mathrm{d}\zeta = \E[ f(\bm{O}\xi) g(\bm{O}\eta) ],
	\end{align}
	which is exactly the definition of $k(t)$ from \eqref{eq:correlation}.
\end{lemma}
\begin{proof}
	We have
	\begin{align*}
		\int_{S^{n-1}} f(\zeta) (A_t g)(\zeta)\,\mathrm{d}\zeta &=
			\int_{SO(n)} f(O\xi) (A_t g)(O\xi)\,\mathrm{d}O\\
			&= \int_{SO(n)} f(O\xi) \int_{(O \xi)^{t}} g(\psi)
				\,\mathrm{d}\sigma_{O\xi, t}(\psi)\,\mathrm{d}O,
	\end{align*}
	If $H$ is the subgroup of all elements in $SO(n)$ which fix $\xi$, then the
	above integral can be rewritten
	\begin{align*}
		\int_{SO(n)} f(O\xi) \int_H g(Oh\eta)
			\,\mathrm{d}h \,\mathrm{d}O.
	\end{align*}
	By Fubini's theorem, this integral is equal to
	\begin{align*}
		&\int_H \int_{SO(n)} f(O\xi) g(Oh\eta) \,\mathrm{d}O \,\mathrm{d}h\\
		=& \int_H \int_{SO(n)} f(O h^{-1} \xi) g(O\eta) \,\mathrm{d}O \,\mathrm{d}h\\
		=& \int_{SO(n)} f(O \xi) g(O\eta) \,\mathrm{d}O,
	\end{align*}
	where we use the right-translation invariance of the Haar integral on $SO(n)$ at the first equality,
	and the second inequality follows by noting that the integrand is constant
	with respect to $h$.
\end{proof}

\begin{lemma}\label{lm:SelfAdj} For every $t\in (-1,1)$, the operator $A_t:L^2(S^{n-1})\to L^2(S^{n-1})$ is self-adjoint.
\end{lemma}
\begin{proof} 
 Fix $\xi,\eta \in S^{n-1}$ that satisfy
	$\langle \xi, \eta \rangle = t$. Lemma~\ref{lm:twoPoint} implies that for any $f,g\in L^2(S^{n-1})$,
we have
	\[
		\langle A_t f, g \rangle = \E_{\bm{O} \in SO(n)}[ f(\bm{O}\xi) g(\bm{O} \eta) ]
			= \langle f, A_t g \rangle,
	\]
 giving the required.\end{proof}

\section{Existence of a measurable maximum independent set}\label{sec:max}

Let $n \geq 2$ and  $X \subset [-1,1]$.
From Theorem~\ref{thm:irCircle} we know that the supremum $\alpha_X(n)$
is sometimes attained as a maximum, and sometimes not.
It is therefore interesting to ask when a maximizer exists. The main positive
result in this direction is Theorem~\ref{thm:attainment}, which says
that a largest measurable $X$-avoiding set always exists when $n \geq 3$.
Remarkably, this result holds under no additional restrictions (not even Lebesgue measurability)
on the set $X$ of forbidden inner products.
Before arriving at this theorem, we shall need to establish a number of technical results.
For the remainder of this section we suppose $n \geq 3$.

For $d \geq 0$, let $\mathrsfs{H}_d^n$ be the vector space of homogeneous
polynomials $p(x_1,\dots,x_n)$ of degree $d$ in $n$ variables belonging to the kernel of the Laplace operator; that is
\[
	\frac{\partial^2 p}{\partial x_1^2} + \cdots + \frac{\partial^2 p}{\partial x_n^2} = 0.
\]
Note that each $\mathrsfs{H}_d^n$ is finite-dimensional.
The restrictions of the elements of $\mathrsfs{H}_d^n$ to the surface of the unit
sphere are called the \emph{spherical harmonics}. For fixed $n$, we
have $L^2(S^{n-1}) = \oplus_{d=0}^\infty \mathrsfs{H}_d^n$
(\cite[Theorem~2.2.2]{dai+xu:athasb}); that is, each function in $L^2(S^{n-1})$ can be
written uniquely as an infinite sum of elements from $\mathrsfs{H}_d^n$, $d=0,1,2,\dots$,
with convergence in the $L^2$-norm.

Recall the definition \eqref{eq:transOperator} of the adjacency operator from
Section \ref{sec:adj}:
\[
	(A_t f)(\xi) := \int_{\xi^t} f(\eta) \,\mathrm{d}\sigma_{\xi,t}(\eta),\quad f\in L^2(S^{n-1}).
\]

The next lemma states that each spherical harmonic is an eigenfunction
of the operator $A_t$. It extends the Funk-Hecke formula (\cite[Theorem~1.2.9]{dai+xu:athasb}) 
to the Dirac measures, obtaining the eigenvalues of $A_t$ explicitly.
The proof relies on the fact that integral kernel operators $K$
having the form $(Kf)(\xi) = \int f(\zeta) k(\langle \zeta, \xi \rangle) \,\mathrm{d}\zeta$
for some function $k : [-1,1] \to \R$
are diagonalised by the spherical harmonics, and moreover that the eigenvalue
of a specific spherical harmonic depends only on its degree.

\begin{proposition}\label{pr:adjEigs}
	Let $t \in (-1,1)$. Then for every spherical harmonic $Y_d$ of degree~$d$,
	\[
		(A_t Y_d)(\xi) = \int_{\xi^t} Y_d(\eta) \,\mathrm{d}\sigma_{\xi,t}(\eta)
			= \mu_d(t) Y_d(\xi), ~~\xi \in S^{n-1},
	\]
	where $\mu_d(t)$ is the constant
	\[
		\mu_d(t) = C_d^{(n-2)/2}(t)
			(1-t^2)^{(n-3)/2}.
	\]
\hide{	\[
		\mu_d(t) = C_d^{(n-2)/2}(t)
			(1-t^2)^{(n-3)/2} \bigg/ C_d^{(n-2)/2}(1).
	\]
}
\end{proposition}
\begin{proof}
	Let $\,\mathrm{d}s$ be the Lebesgue measure on $[-1,1]$ and
	let $\{ f_\alpha \}_\alpha$ be a net of functions in $L^1([-1,1])$
	such that $\{ f_\alpha \,\mathrm{d}s \}$ converges to the Dirac point
	mass $\delta_t$ at $t$ in the weak-* topology on the set of Borel measures on $[-1,1]$.
	By Theorem~1.2.9 in \cite{dai+xu:athasb}, we have
	\[
		\int_{S^{n-1}} Y_d(\eta) f_\alpha(\langle \xi, \eta \rangle) \,\mathrm{d}\eta = \mu_{d,\alpha} Y_d(\xi),
	\]
	where
	\[
		\mu_{d,\alpha} =  \int_{-1}^1 C_d^{(n-2)/2}(s) (1-s^2)^{(n-3)/2} f_\alpha(s) \,\mathrm{d}s.
	\]
\hide{	\[
		\mu_{d,\alpha} =  \int_{-1}^1 C_d^{(n-2)/2}(s) (1-s^2)^{(n-3)/2} f_\alpha(s) \,\mathrm{d}s
			\bigg/ C_d^{(n-2)/2}(1).
	\]
}

	By taking limits, we finish the proof.
\end{proof}

The next lemma is a general fact about weakly convergent sequences in
a Hilbert space.

\begin{lemma}\label{lm:weakCompact}
	Let $\mathcal{H}$ be a Hilbert space and let $K : \mathcal{H} \to \mathcal{H}$
	be a compact operator. Suppose $\{ x_i \}_{i=1}^\infty$ is a sequence in
	$\mathcal{H}$ converging weakly to $x \in \mathcal{H}$. Then
	\[
		\lim_{i \to \infty} \langle K x_i, x_i \rangle = \langle K x, x \rangle.
	\]
\end{lemma}
\begin{proof}
	Let $C$ be the maximum of $\| x \|$  and $\sup_{i \geq 1} \| x_i \| $,
	which is finite by the Principle of Uniform Boundedness.
	Let $\{ K_m \}_{m=1}^\infty$ be a sequence of finite rank operators such that
	$K_m \to K$ in the operator norm as $m \to \infty$. Clearly
	\[
		\lim_{i \to \infty} \langle K_m x_i, x_i \rangle = \langle K_m x,x \rangle
	\]
	for each $m=1,2,\dots$\ . Let $\e > 0$ be given and choose $m_0$ so that
	$\| K - K_{m_0} \| < \e/(3C^2)$. Choosing $i_0$ so that
	$| \langle K_{m_0} x_i, x_i \rangle - \langle K_{m_0} x, x \rangle | < \e/3$
	whenever $i \geq i_0$, we have
	\begin{align*}
		& | \langle K x_i, x_i \rangle - \langle K x, x \rangle | \\
		\leq&  | \langle K x_i, x_i \rangle - \langle K_{m_0} x_i, x_i \rangle |
			+ | \langle K_{m_0} x_i, x_i \rangle - \langle K_{m_0} x, x \rangle |
			+ | \langle K_{m_0} x, x \rangle - \langle Kx, x \rangle | \\
		\leq& \| K - K_{m_0} \| C^2 + \e/3 + \| K - K_{m_0} \| C^2 \ <\ \e,
	\end{align*}
	and the lemma follows.
\end{proof}

The next corollary is also a result stated in \cite{bachoc13}.

\begin{corollary}\label{cor:compact}
	If $n \geq 3$ and $t \in (-1,1)$, then $A_t$ is compact.
\end{corollary}
\begin{proof}
	The operator $A_t$ is diagonalisable by Proposition~\ref{pr:adjEigs},
	since the spherical harmonics
	form an orthonormal basis for $L^2(S^{n-1})$. It therefore suffices
	to show that its eigenvalues cluster only at $0$.
		
	By Theorem~8.21.8 of \cite{szego92} and Proposition~\ref{pr:adjEigs}, the eigenvalues
	$\mu_d(t)$ tend to zero as $d \to \infty$. The eigenspace corresponding
	to the eigenvalue $\mu_d(t)$ is precisely the vector space of
	spherical harmonics of degree $d$, which is finite dimensional.
	Therefore $A_t$ is compact.
\end{proof}


For each $\xi \in S^{n-1}$, let $C_h(\xi)$ be the open spherical cap
of height $h$ in $S^{n-1}$ centred at $\xi$. Recall that $C_h(\xi)$
has volume proportional to $\int_{1-h}^1 (1-t^2)^{(n-3)/2} \,\mathrm{d}t$.

\begin{lemma}\label{lm:caps}
	For each $\xi \in S^{n-1}$, we have $\lambda(C_h(\xi)) = \Theta(h^{(n-1)/2})$, and\\
	$\lambda(C_{h/2}(\xi)) \geq \lambda(C_{h}(\xi))/2^{(n-1)/2} - o(h^{(n-1)/2})$
	as $h \to 0^+$.
\end{lemma}
\begin{proof}
	If $f(h) = \int_{1-h}^1 (1-t^2)^{(n-3)/2} \,\mathrm{d}t$, then
	we have $\frac{df}{dh}(h) = (2h-h^2)^{(n-3)/2}$.
	Since $f(0) = 0$, the smallest power of $h$ occurring in $f(h)$ is of order
	$(n-1)/2$. This gives the first result.
	For the second, note that the coefficient
	of the lowest order term in $f(h)$ is $2^{(n-1)/2}$
	times that of $f(h/2)$.
\end{proof}

\begin{lemma}\label{lm:density}
	Suppose $n \geq 3$ and
	let $I \subset S^{n-1}$ be a Lebesgue measurable set with $\lambda(I) > 0$.
	Define $k:[-1,1] \to \R$ by
	\[
		k(t) := \int_{S^{n-1}} \one_I(\zeta) (A_t \one_I)(\zeta) \,\mathrm{d}\zeta,
	\]
	which by Lemma~\ref{lm:twoPoint} is the same as Definition \eqref{eq:correlation}
	applied with $f=g=\one_I$.
	If $\xi_1, \xi_2 \in S^{n-1}$ are Lebesgue density points of $I$,
	then $k(\langle \xi_1, \xi_2 \rangle) > 0$.
\end{lemma}
\begin{proof}
	Let $t = \langle \xi_1, \xi_2 \rangle$.
	If $t = 1$, then the conclusion holds since $k(1)=\lambda(I)>0$.
	If $t=-1$, then $\xi_2 = -\xi_1$, and by the Lebesgue density theorem
	we can choose $h>0$ small enough that
	$\lambda(C_h(\xi_i) \cap I) > \frac{2}{3} \lambda(C_h(\xi_i))$ for $i=1,2$.
	By Lemma~\ref{lm:twoPoint} we have
	\begin{align*}
		k(-1) &= \E[\one_I(\bm{O}\xi_1) \one_I(\bm{O}(-\xi_1))]\\
		&\geq \E[\one_{I \cap C_h(\xi_2) }(\bm{O}\xi_1) \one_{I \cap C_h(\xi_2) }(\bm{O}(-\xi_1))]
		\geq \frac{1}{3}\lambda(C_h(\xi_1)).
	\end{align*}
	
	From now on we may therefore assume $-1 < t < 1$.
	Let $h>0$ be a small number which
	will be determined later. Suppose $x \in C_h(\xi_1)$. The intersection
	$x^t \cap C_h(\xi_2)$ is a spherical cap in the $(n-2)$-dimensional sphere
	$x^t$ having height proportional to $h$; this is because $C_h(\xi_2)$ is
	the intersection of $S^{n-1}$ with a certain halfspace $H$, and
	$x^t \cap C_h(\xi_2) = x^t \cap H$.
	We have $\sigma_{x,t}(x^t \cap C_h(\xi_2)) = \Theta(h^{(n-2)/2})$ by
	Lemma~\ref{lm:caps}, and it follows
	that there exists $D>0$ such that $\sigma_{x,t}(x^t \cap C_h(\xi_2)) \leq Dh^{(n-2)/2}$
	for sufficiently small $h>0$.
	
	If $x \in C_{h/2}(\xi_1)$, then $x^t \cap C_{h/2}(\xi_2) \neq \emptyset$
	since $x^t$ is just a rotation of the hyperplane $\xi_1^t$ through an
	angle equal to the angle between $x$ and $\xi_1$.
	Therefore $x^t \cap C_h(\xi_2)$ is a spherical cap
	in $x^t$ having height at least $h/2$.

 Thus there exists $D' > 0$ such that
	$\sigma_{x,t}(x^t \cap C_h(\xi_2)) \geq D' h^{(n-2)/2}$ for all $x \in C_{h/2}(\xi_1)$,
	by Lemma~\ref{lm:caps}.
	
	Now choose $h>0$ small enough that
	$\lambda(C_h(\xi_i) \cap I) \geq (1 - \frac{D'}{2^n D}) \lambda(C_h(\xi_i))$ for $i=1,2$;
	this is possible by the Lebesgue density theorem
	since $\xi_1$ and $\xi_2$ are density points.
	We have by Lemma~\ref{lm:twoPoint} that
	\[
		k(t) = \prob[\bm{\eta}_1 \in I, \bm{\eta}_2 \in I],
	\]
	if $\bm{\eta}_1$ is chosen uniformly at random from $S^{n-1}$, and if
	$\bm{\eta}_2$ is chosen uniformly at random from $\bm{\eta}_1^t$.
	Then
	\begin{align*}
		k(t) &\geq \prob[\bm{\eta}_1 \in I \cap C_h(\xi_1), \bm{\eta}_2 \in I \cap C_h(\xi_2)]\\
			&\geq \prob[\bm{\eta}_1 \in C_h(\xi_1), \bm{\eta}_2 \in C_h(\xi_2)]
			-\prob[\bm{\eta}_1 \in C_h(\xi_1) \setminus I, \bm{\eta}_2 \in C_h(\xi_2) ]\\
			&~~~~~~~~~~~~~~-\prob[\bm{\eta}_1 \in C_h(\xi_1), \bm{\eta}_2 \in C_h(\xi_2) \setminus I].
	\end{align*}
	
	The first probability is at least
	\begin{align*}
		D'  h^{(n-2)/2} \lambda(C_{h/2}(\xi_1))
		\geq \frac{D'}{2^{(n-1)/2}} h^{(n-2)/2} \lambda(C_h(\xi_1)) - o(h^{(2n-3)/2})
	\end{align*}
	by Lemma~\ref{lm:caps}. The second and third probabilities are each no more than
	$$
 \frac{D'}{2^n D} \lambda(C_h(\xi_1)) D h^{(n-2)/2}
	= \frac{D'}{2^n} \lambda(C_h(\xi_1)) h^{(n-2)/2}
 $$ for
	sufficiently small $h>0$, and therefore by the first part
	of Lemma~\ref{lm:caps},
	\[
		k(t) \geq \frac{D'}{2^{(n-1)/2}} \lambda(C_h(\xi_1)) h^{(n-2)/2} - o(h^{(2n-3)/2})
			- \frac{D'}{2^{n-1}} \lambda(C_h(\xi_1)) h^{(n-2)/2},
	\]
	and this is strictly positive for sufficiently small $h>0$.
\end{proof}

We are now in a position to prove the second main result of this paper.

\begin{theorem}\label{thm:attainment}
	Suppose $n \geq 3$ and let $X$ be any subset of $[-1,1]$.
	Then there exists an $X$-avoiding set $I \in \mathcal{L}$
	such that $\lambda(I) = \alpha_X(n)$.
\end{theorem}
\begin{proof}
	We may suppose that $1 \not\in X$ for otherwise every $X$-avoiding set
	is empty and the theorem holds with $I=\emptyset$. 

Let $\{ I_i \}_{i=1}^\infty$ be a sequence of measurable $X$-avoiding sets
	such that $\lim_{i \to \infty} \lambda(I_i) = \alpha_X(n)$. Passing to a
	subsequence if necessary, we may suppose that the sequence
	$\{ \one_{I_i} \}$ of characteristic functions converges weakly in
	$L^2(S^{n-1})$; let $h$ be its limit. Then $0 \leq h \leq 1$ almost everywhere
	since $0 \leq \one_{I_i} \leq 1$ for every $i$.
	
	Denote by $I'$ the set $h^{-1}((0,1])$, and let $I$ be the set of Lebesgue
	density points of $I'$. We claim that $I$ is $X$-avoiding.

	For all $t \in X \setminus \{ -1 \}$,
	the operator $A_t : L^2(S^{n-1}) \to L^2(S^{n-1})$ is self-adjoint and compact
	by Lemma~\ref{lm:SelfAdj} and Corollary~\ref{cor:compact}.
	Since $\langle A_t \one_{I_i}, \one_{I_i} \rangle = 0$ for each $i$,
	Lemma~\ref{lm:weakCompact} implies $\langle A_t h, h \rangle = 0$.
	Since $h \geq 0$, it follows from the definition of $A_t$ that
	$\langle A_t \one_{I'}, \one_{I'} \rangle = 0$,
	and therefore also that $\langle A_t \one_{I}, \one_{I} \rangle = 0$.
	But if there exist points $\xi, \eta \in I$ with
	$t_0 = \langle \xi, \eta \rangle \in X \setminus \{ -1 \}$, then
	$\langle A_{t_0} \one_{I}, \one_{I} \rangle > 0$ by Lemma~\ref{lm:density}.

Thus, in order to show that $I$ is $X$-avoiding, it remains to derive a contradiction from assuming that $-1\in X$ and $-\xi,\xi\in I$ for some $\xi\in S^{n-1}$. Since $\xi$ and $-\xi$ are Lebesgue density points of $I$, there is a spherical cap $C$ centred at $\xi$
such that $\lambda(I \cap C) > \frac{2}{3} \lambda(C)$ and
$\lambda(I \cap (-C)) > \frac{2}{3} \lambda(C)$. The same applies to $I_i$ for all large $i$ 
(since a cap
is a continuity set). But this contradicts the fact that
$I_i$ and its reflection $-I_i$ are disjoint for every $i$. Thus $I$ is $X$-avoiding.

	Finally, we have
	\begin{align*}
		\lambda(I)  &= \lambda(I')
			\geq \langle \one_{S^{n-1}}, h \rangle
			= \lim_{i \to \infty} \langle \one_{S^{n-1}}, \one_{I_i} \rangle
			= \lim_{i \to \infty} \lambda(I_i)
			= \alpha_X(n),
	\end{align*}
	whence $\lambda(I) = \alpha_X(n)$ since $\lambda(I) \leq \alpha_X(n)$.
\end{proof}

\section{Invariance of $\alpha_X(n)$ under taking the closure of $X$}\label{sec:closure}

Again let $n \geq 2$ and $X \subset [-1,1]$. We will use $\overline{X}$ to
denote the toplogical closure of $X$ in $[-1,1]$.
In general it is false that $X$-avoiding sets are $\overline{X}$-avoiding.
In spite of this, we have the following result.

\begin{theorem}\label{thm:closureIR}
	Let $X$ be an arbitrary subset of $[-1,1]$.
	Then $\alpha_X(n) = \alpha_{\overline{X}}(n)$.
	In particular $\alpha_X(n) = 0$ if $1 \in \overline{X}$.
\end{theorem}
\begin{proof}
	Clearly $\alpha_X(n) \geq \alpha_{\overline{X}}(n)$
	For the reverse inequality,
	let $I' \subset S^{n-1}$ be any measurable $X$-avoiding set. Let
	$I \subset I'$ be the set of Lebesgue density points of $I'$,
	and define $k:[-1,1] \to \R$ by
	$k(t) = \int_{S^{n-1}} \one_I(\zeta) (A_t \one_I)(\zeta) \,\mathrm{d}\zeta$.
	Then $k$ is continuous by Lemma~\ref{lm:pt} and Lemma~\ref{lm:twoPoint},
	and since $k(t) = 0$ for every
	$t \in X$, it follows that $k(t) = 0$ for every $t \in \overline{X}$.
	Lemma~\ref{lm:density} now implies that $I$ is $\overline{X}$-avoiding.
	The theorem now follows since $I'$ was arbitrary, and $\lambda(I) = \lambda(I')$
	by the Lebesgue density theorem.
\end{proof}

\section{Single forbidden inner product}\label{sec:single}

An interesting case to consider is when $|X|=1$, motivated by the fact that $1/\alpha_{\{t\}}(n)$ is
a lower bound on the measurable chromatic number of $\I R^n$ for any $t\in(-1,1)$ and  this freedom
of choosing $t$ may lead to better bounds.

Let us restrict ourselves to the special case when $n=3$ (that is, we look at the 2-dimensional sphere).
For a range of $t\in [-1,\cos \frac{2\pi}5]$, the best construction that we could find consists of one or two spherical caps as follows. Given $t$, let $h$ be the maximum height of an open spherical cap which is $\{t\}$-avoiding. A simple calculation shows that $h=1-\sqrt{(t+1)/2}$. If $t\le -1/2$, then we just
take a single cap $C$ of height $h$, which gives that $\alpha_{\{t\}}(3)\ge h/2$ then. When $-1/2< t\le 0$, 
we can add another cap $C'$ whose centre is opposite to that of $C$. When $t$ reaches $0$, 
the caps $C$ and $C'$ have the same height (and we get the two-cap construction from Kalai's conjecture). 
When $0<t\le \frac{2\pi}5$, we can form a $\{t\}$-avoiding set by taking two caps of the same height $h$. 
(Note that the last construction cannot be optimal for $t>\frac{2\pi}5$, as then the two caps can be arranged
so that a set of positive measure can be added, see the third picture of Figure~\ref{fg:1}.)

\begin{figure}
\begin{center}
\includegraphics[height=4cm]{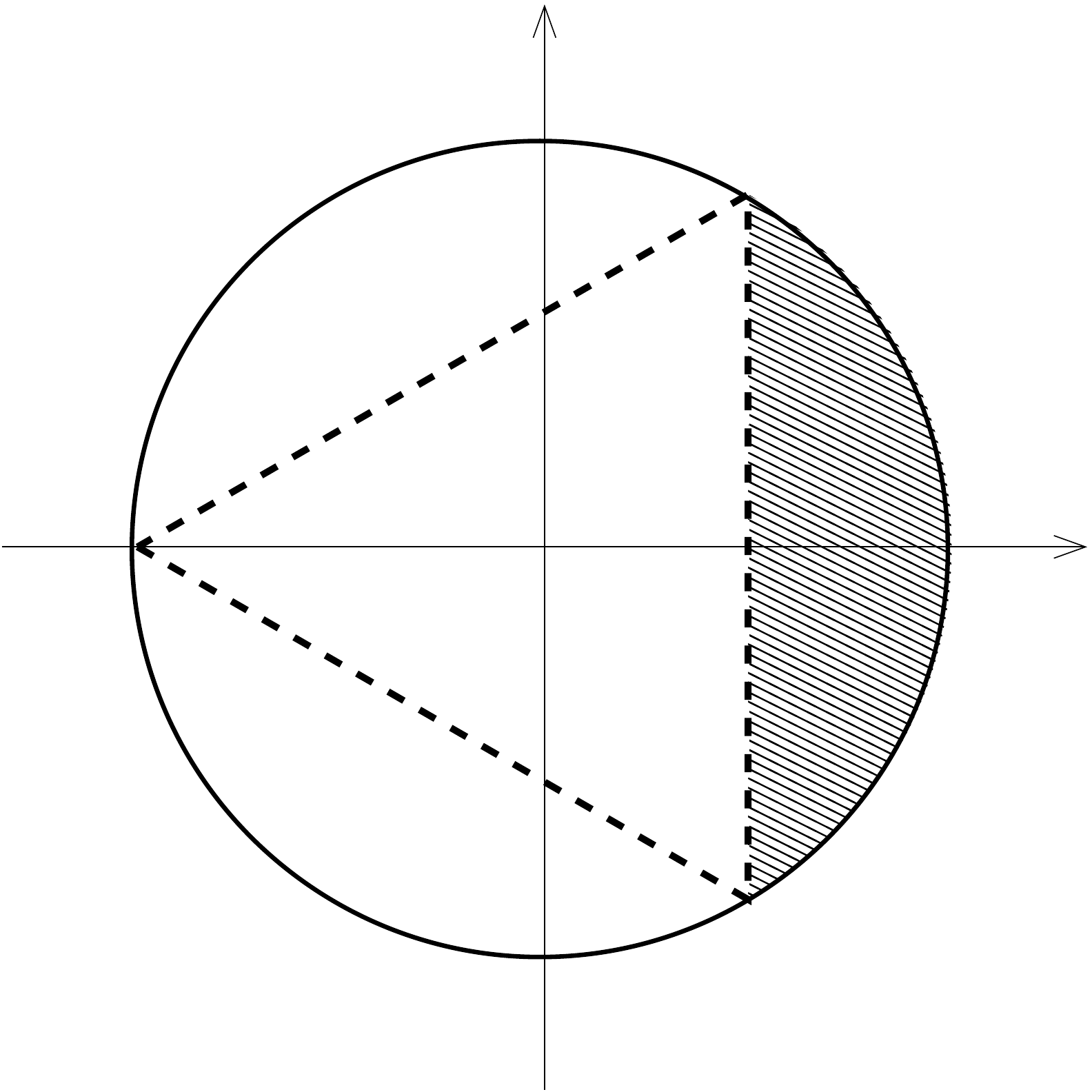}\hspace{1cm}
\includegraphics[height=4cm]{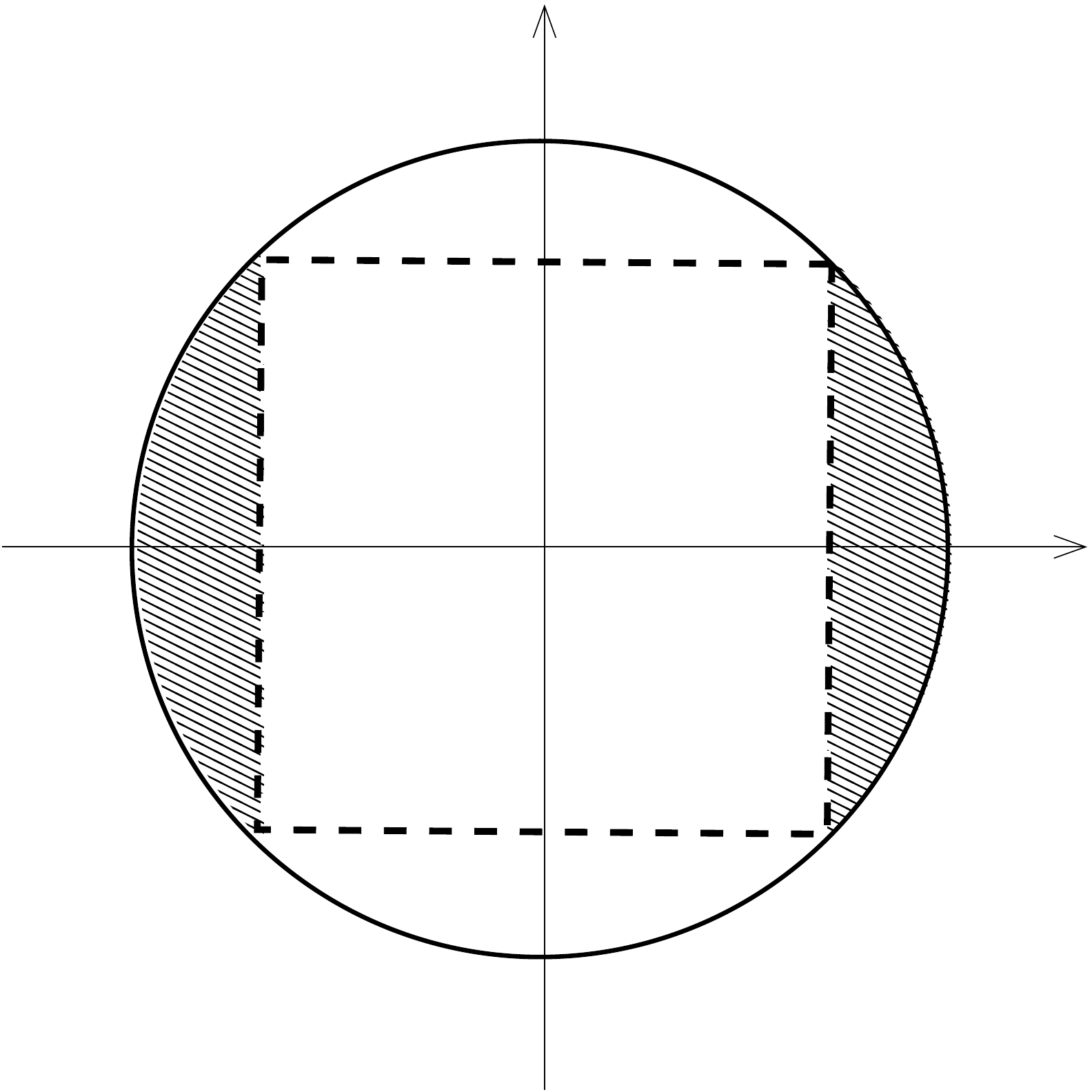}\hspace{1cm}
\includegraphics[height=4cm]{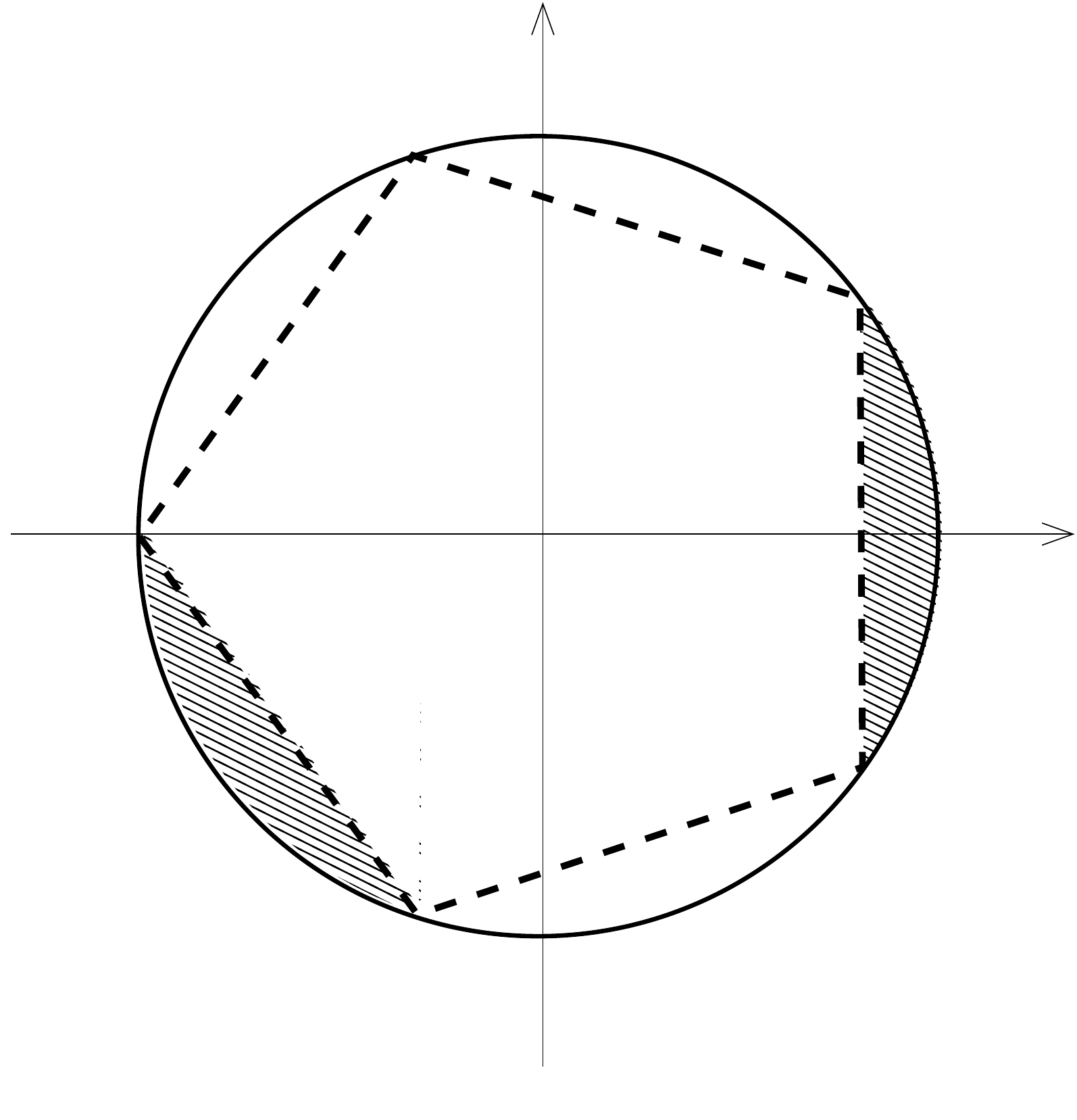}
\end{center}
\caption{$\{t\}$-Avoiding set for $t=-\frac12$, $0$ and $\cos\frac{2\pi}5$}
\label{fg:1}
\end{figure}

Calculations show that the above construction gives the following lower bound (where $h=1-\sqrt{(t+1)/2}$):
 \begin{equation}\label{eq:lower}
 \alpha_{\{t\}}(3)\ge \left\{\begin{array}{ll}
\frac h2,&  -1\le t\le -\frac12,\\
h+t-ht,&  -\frac12\le t\le 0,\\
h,& 0\le t\le \cos \frac{2\pi}5.
 \end{array}\right.
\end{equation}

We conjecture.that the bounds in (\ref{eq:lower}) are all equalities. In particular, our conjecture states that, 
for $t\le -1/2$, we can strengthen Levy's isodiametric inequality by forbidding a single inner product $t$ instead of the whole interval $[-1,t]$. 

As in Section~\ref{sec:lp+comb}, one can write an infinite linear program that gives an upper bound on $\alpha_{\{t\}}(3)$. Although our numerical experiments indicate that the upper bound
given by the LP exceeds the lower bound in (\ref{eq:lower}) by at most $0.062$ for all $-1\le t\le 0.3$, we were not able to 
determine the exact value of $\alpha_{\{t\}}(3)$ for any single $t\in(0,\cos \frac{2\pi}5]$.

\section*{Acknowledgements}

Both authors acknowledge Anusch Taraz and his research group
for their hospitality during the summer
of 2013. The first author would like to thank his thesis advisor Frank Vallentin
for careful proofreading, and for pointing him to the Witsenhausen problem \cite{witsenhausen74}.

\bibliographystyle{alpha} 
\bibliography{opp-14jan2015}

\end{document}